\documentclass[envcountsame,envcountsect]{svmult}

\usepackage{mathptmx}       
\usepackage{helvet}         
\usepackage{courier}        
\usepackage{type1cm}        
\usepackage{makeidx}         
\usepackage{graphicx}        
\usepackage{multicol}        
\usepackage[bottom]{footmisc}

\makeindex             

\usepackage{amstext,amsfonts,amssymb,amscd,amsbsy,amsmath}
\usepackage{tikz-cd}	
\usepackage{tikz}

\newcommand{\A}{\mathbb{A}}

\newcommand{\PP}{\mathbb{P}}

\newcommand{\RR}{\mathbb{R}}

\DeclareMathOperator{\Gr}{Gr}

\DeclareMathOperator{\trop}{trop}

\DeclareMathOperator{\Spec}{Spec}

\newcommand{\cO}{\mathcal{O}}

\smartqed

\usepackage{booktabs}
\usepackage{epigraph}
\tikzset{>=latex}
\usetikzlibrary{calc}

\newcommand{\CC}{\mathbb{C}}%
\newcommand{\kk}{\Bbbk}%
\newcommand{\bQ}{\mathbf{Q}}%

\DeclareMathOperator{\Hom}{\operatorname{Hom}}%
\DeclareMathOperator{\Proj}{\operatorname{Proj}}%

\begin{document}

\title*{Equations and tropicalization of Enriques surfaces}
\author{Barbara Bolognese, Corey Harris and Joachim Jelisiejew}
\institute{Barbara Bolognese \at The University of Sheffield   \email{b.bolognese@sheffield.ac.uk}
\and Corey Harris \at Florida State University \email{charris@math.fsu.edu}
\and Joachim Jelisiejew \at Institute of Mathematics, Informatics and Mechanics, University of Warsaw \email{jjelisiejew@mimuw.edu.pl}}
%
%
\maketitle

\emph{``Ogni superficie $F$ di generi $p_a=p_g=P_3=0, P_2=1$ si pu\`o
            considerare come una superficie doppia di generi $1$. Pi\`u
            precisamente: le coordinate dei punti di $F$ si possono esprimere
            razionalmente per mezzo di $x,y,z$ e di $z=\sqrt{f_s(x,y)}$, dove
        il polinomio $f_s$ \`e del quarto grado separatamente rispetto ad
    $x,y$, ed ammette una trasformazione involutoria in se stesso priva di
coincidenze; ad ogni punto di $F$ corrispondono due terne $(x,y,z)$.”}-- F.
Enriques~\cite{enriques1908}\newline

\vspace{0.2in}

\abstract{
    In this article we explicitly compute equations of an Enriques
    surface via the involution on a K3 surface. We also discuss its tropicalization and compute the tropical homology, thus recovering a special case of the result of \cite{IKMZ}, and establish a connection between the dimension of the tropical homology groups and the Hodge numbers of the corresponding algebraic Enriques surface.}

\section{Introduction}

In the classification of algebraic surfaces, Enriques surfaces comprise one of
four types of minimal surfaces of 
Kodaira dimension $0$. There are a number of surveys on Enriques surfaces. For those new to the theory, 
we recommend the excellent exposition found in~\cite{bhpv} and~\cite{beauville_book}, and 
for a more thorough treatment, the book~\cite{cossec-dolgachev1989}. 
Another recommended source is~Dolgachev's brief introduction to
Enriques surfaces~\cite{Dolgachev_intro}.

The first Enriques surface was constructed in 1896 by Enriques
himself~\cite{enriques1896} to answer negatively a question posed by Castelnuovo (1895):
\begin{center}
Is
every surface with $p_g=q=0$ rational?
\end{center}
(see Section~\ref{generalities} for the meaning of $p_q$ and $q$)
    Enriques' original surface has a
beautiful geometric construction: the normalization of a degree 6 surface in
$\PP^3$ with double lines given by the edges of a tetrahedron. Another construction, 
the Reye congruence, defined a few years earlier by Reye~\cite{reye1882}, was later proved by Fano~\cite{fano1901} 
to be an Enriques surface. Since these first constructions, there have been many examples of 
Enriques surfaces, most often as quotients of K3 surfaces by a fixed-point-free involution. 
In~\cite{C83}, Cossec describes all birational models of Enriques surfaces given by complete 
linear systems. 

As we recall in Section~\ref{generalities}, every Enriques surface has an unramified double cover 
given by a K3 surface. Often exploiting this double cover, topics of particular interest relate to 
lattice theory, moduli spaces and their compactifications, automorphism groups of Enriques surfaces, 
and Enriques surfaces in characteristic 2. 

While there are many constructions of Enriques surfaces, none give explicit equations for an
Enriques surface embedded in a projective space. In this paper, interpreting
the work of Cossec-Verra, we give explicit ideals for all
Enriques surfaces. 

\begin{theorem}\label{ref:mainthm1}
Let $Y$ be the toric fivefold of degree 16 in $\PP^{11}$ that is
obtained by taking the join of the Veronese surface in $\PP^5$ with itself.
The intersection of $Y$ with a general linear subspace of codimension 3
is an Enriques surface, and every Enriques surface arises in this way.
\end{theorem}
By construction, the Enriques surface in Theorem \ref{ref:mainthm1} is arithmetically
Cohen-Macaulay. Its homogeneous prime ideal in
the polynomial ring with 12 variables is generated by the twelve binomial
quadrics that define $Y$ and three additional linear forms. Explicit code
for producing this Enriques ideal in Macaulay 2 is given in Section \ref{sec:constructions}.


After having constructed Enriques surfaces explicitly, we focus on their
tropicalizations, with the purpose of studying their combinatorial properties.
For this we choose a different K3 surface, namely a hypersurface $S \subset
(\PP^1)^3$ with an involution $\sigma$, see~Example~\ref{ref:kristinsK3:exam}.
In Section~\ref{sec:K3} we get a fairly complete picture for its
tropicalization.
In particular we recover its Hodge numbers and, conjecturally, the Hodge
numbers of $S/\sigma$, which was \cite[Problem~10 on
Surfaces]{Bernd}; this was the starting point of this work.
\begin{proposition}[Example~\ref{ref:kristinsK3:exam}, Proposition~\ref{ref:hodgeagreeK3:prop}, Proposition~\ref{ref:hodgeagreeEnriques:prop}]\label{ref:descriptionoftropK3:prop}
    The dimensions of tropical homology groups of the tropicalization of the K3 surface $S$
    agree with Hodge numbers of $S$. The dimensions of $\sigma$-invariant parts
    of tropical homology groups agree with the Hodge numbers of the Enriques surface $S/\sigma$.
\end{proposition}
\begin{figure}[!ht]\label{K3trop}
  \centering
    \includegraphics[width=0.45\textwidth]{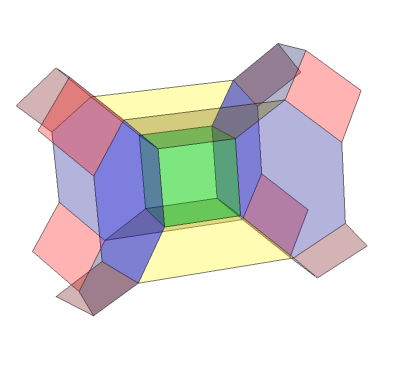}
    \caption{A tropical K3 in $\PP ^1\times \PP^1 \times \PP^1$ that is fixed under the involution}
\end{figure}

Finally we discuss an analogue of Castelnuovo's question on the tropical and
analytic level. Since the analytifications of rational varieties are
contractible by \cite[Corollary 1.1.4]{BrownFoster}, we ask the following question:
\begin{center}
    Are the analytifications of K3 or Enriques surfaces contractible?
\end{center}
We give a negative answer to this question, the counterexample being the analytification of $S$ from
Example~\ref{ref:kristinsK3:exam}.
\begin{theorem}\label{ref:topologyOfAnalytifications:theorem}
    The analytification $S^{an}$ of the K3 surface $S$ is homotopy equivalent to a
    two-dimensional sphere. The surface $S$ has a fixed-point-free involution
    $\sigma$ and the analytification of the Enriques surface $S/\sigma$
    retracts onto $\RR\PP^2$. In particular neither $S^{an}$ nor
    $(S/\sigma)^{an}$ is contractible.
\end{theorem}

The contents of the paper are as follows.
In Section~\ref{generalities} we give some background about Enriques
surfaces. Next, in Section~\ref{sec:constructions}, we exploit a classical construction
to obtain an Enriques ideal in a codimension $3$ linear space in $\PP
^{11}$ and prove Theorem~\ref{ref:mainthm1}.
In Section~\ref{sec:tropicalization} we discuss the basics of tropical
geometry and analytic spaces in the sense of Berkovich.
Example~\ref{ref:kristinsK3:exam} provides an Enriques surface $S/\sigma$ arising from a
K3 surface $S \subset \PP^1\times \PP^1\times \PP^1$ with an involution $\sigma$.
The surface $S$ is suitable from the tropical point of view (its tropical variety
is sch\"on and multiplicity one everywhere) and is used throughout the paper.
In Section~\ref{Homology}, we compute the tropical homology groups of
$\trop(S)$ and, conjecturally, of $\trop(S/\sigma)$. We also prove
Proposition~\ref{ref:descriptionoftropK3:prop}.
In Section~\ref{sec:analytic} we discuss the topology of analytifications of
$S$ and $S/\sigma$ and prove
Theorem~\ref{ref:topologyOfAnalytifications:theorem}.

\section{Background}\label{generalities}

Apart from the code snippets, we work over an algebraically closed field $\kk$ of
characteristic zero.
An \emph{Enriques surface} $X$
is a smooth projective surface with $q(X) := h^1(X, \mathcal{O}_X) = 0$,
$\omega_X^{\otimes 2}  \simeq  \mathcal{O}_{X}$ and $\omega_{X} \not\simeq
\mathcal{O}_X$, where $\omega_X=\bigwedge^2\Omega^1_X$ is the canonical bundle
of $X$. Then it follows that $X$ is minimal, see \cite{beauville_book}, and
$p_g(X):=h^2(X, \mathcal{O}_X)=0$. We note that Enriques surfaces are defined
the same way over any field of characteristic other than $2$.
By~\cite[Lemma~15.1]{bhpv} the Hodge diamond of an Enriques surface $X$ is:
\begin{equation}
    \begin{tikzpicture}
        \matrix (m) [matrix of math nodes, row sep=0em, column sep=0.5em]
        {
&&h^{0,0}&& & &&&1&&\\
&h^{1,0}&&h^{0,1}& & &&0&&0&\\
h^{2,0}&&h^{1,1}&&h^{0,2} & = & 0 &&10&&0\\
&h^{2,1}&&h^{1,2}& & &&0&&0&\\
&&h^{2,2}&& & &&&1&&.\\
        };
    \end{tikzpicture}
\end{equation}

An Enriques surface admits an unramified double cover
$f:Y\to X$, where $Y$ is a K3 surface, see~\cite[Lemma~15.1]{bhpv}
or~\cite[Proposition~VIII.17]{beauville_book}.
The Hodge diamond of $Y$ is given by

\begin{equation}
    \begin{tikzpicture}
        \matrix (m) [matrix of math nodes, row sep=0em, column sep=0.5em]
        {
&&h^{0,0}&& & &&&1&&\\
&h^{1,0}&&h^{0,1}& & &&0&&0&\\
h^{2,0}&&h^{1,1}&&h^{0,2} & = & 1 &&20&&1\\
&h^{2,1}&&h^{1,2}& & &&0&&0&\\
&&h^{2,2}&& & &&&1&&.\\
        };
    \end{tikzpicture}
\end{equation}

In particular since $Y$ is simply-connected, the
fundamental group of an Enriques surface is $\mathbb{Z}/2\mathbb{Z}$, see
\cite[Section~15]{bhpv}.
The cover $Y\to X$ is in fact a quotient of $Y$ by an involution $\sigma$, which
exchanges the two points of each fiber.
Conversely, for a K3 surface $Y$ with a fixed-point-free involution $\sigma$
the quotient $Y/\sigma$ is an Enriques surface. An example of this procedure,
known as Horikawa's construction, appears in the quote at the beginning of the
paper.

\section{Enriques surfaces via K3 complete intersections in $\PP^5$}\label{sec:constructions}

    In this section we construct Enriques surfaces via K3
    surfaces in $\PP^5$.
    Before we go into the details, we remark that one cannot hope for easy
    equations, for example an Enriques surface cannot be a hypersurface in
    $\mathbb{P}^3$.
    \begin{proposition}\label{ref:nosimpleequations:prop}
        Let $X \subset \PP^N_{\CC}$ be a smooth projective toric threefold and $S =
        X\cap H$ be a smooth hyperplane section. Then $S$ is simply-connected. In
        particular it is not an Enriques surface.
    \end{proposition}
    \begin{proof}
        Since $X$ is a smooth and projective toric variety, it is simply
        connected by~\cite[\S3.2]{Fulton_toricvars}. Now a homotopical version
        of Lefschetz' theorem (\cite{barthlarsen} see also
        \cite[2.3.10]{sommese_adjunction}) asserts that the fundamental groups of $X\cap H$
        and $X$ are isomorphic via the natural map. Thus $S$ is
        simply connected.
        Now suppose $S$ is an Enriques surface. Then it admits a non-trivial \`etale
        double cover $K\to S$, thus it is not simply connected, which
        is a contradiction.\qed
    \end{proof}
    We remark that this proof generalizes to other complete intersections
    inside smooth toric varieties, provided that intermediate complete
    intersections are smooth as well.

    \medskip

    We now construct an Enriques
    surface from a K3 surface which is an intersection of quadrics in
    $\PP^5$. We follow Beauville~\cite[Example VIII.18]{beauville_book}.

    Fix a projective space $\mathbb{P}^5$ with coordinates $x_0$, $x_1$, $x_2$,
    $y_0$, $y_1$, $y_2$.
    Consider the involution $\sigma:\mathbb{P}^5\to \mathbb{P}^5$ given by
    $\sigma(x_i) = x_i$ and $\sigma(y_i) = -y_i$ for $i=0,1,2$. Then the fixed
    point set is equal to the union of $\mathbb{P}^2 = V(y_0, y_1, y_2)$ and
    $\mathbb{P}^2 = V(x_0, x_1, x_2)$.

    Fix quadrics $F_i\in \CC[x_0, x_1, x_2]$ and $G_i\in \CC[y_0, y_1, y_2]$,
    where $i=0, 1, 2$ and denote $Q_i := F_i + G_i$. By their construction,
    these quadrics are fixed by $\sigma$. We henceforth choose $Q_i$ so that
    they give a complete intersection.
    Then $S = S_{\bQ}:= V(Q_0, Q_1, Q_2)$ is a surface and, 
    by the Adjunction Formula, we have $K_S = \cO_S(-6+2+2+2) = \cO_S$. 
    It can also be shown that since the surface $S$ is a complete intersection 
    of quadrics in $\PP^5$,  
    it has $h^1(\cO_S) = 0$, see \cite[Lemma VIII.9]{beauville_book}. Thus
    if $S$ is smooth, then it is a K3 surface fixed under the involution $\sigma$.
    We will now formalize exactly which assumptions must be satisfied by 
    the three quadrics to obtain a smooth Enriques surface.
    \begin{definition}
        Let $\bQ = (Q_0, Q_1, Q_2)$ be a triple of quadrics $Q_i = F_i + G_i$
        for $F_i\in \CC[x_i]$ and $G_i\in \CC[y_i]$ as before.
        We say that the quadrics $\bQ$ are
        \emph{enriquogeneous}  if the following conditions
        are satisfied:
        \begin{enumerate}
            \item the forms $\bQ = (Q_0, Q_1, Q_2)$ are a complete
                intersection,
            \item the surface $S = V(Q_0, Q_1, Q_2)$ is smooth,
            \item the surface $S = V(Q_0, Q_1, Q_2)$ does not intersect the fixed-point set of
                $\sigma$.
        \end{enumerate}
    \end{definition}
    We note that the third condition is equivalent to $F_1,F_2,F_3$ having no common
    zeros in $\CC[x_0, x_1, x_2]$ and $G_i$ having no common zeros in $\CC[y_0,
    y_1, y_2]$, so it is open.
    We know that for a choice of enriquogeneous quadrics $\bQ$ we obtain an Enriques
    surface as $S_{\bQ}/\sigma$.  The set of enriquogeneous quadrics is open
    inside $(\A^6)^6$, so that a \emph{general} choice of forms gives
    an Enriques surface.  In~\cite{C83}, Cossec proves that every complex
    Enriques surface may be obtained as above if one allows $\bQ$ not satisfying the smoothness condition, see also~\cite{Verra_doubleCover}.
    Notably, Lietdke recently proved that the same is true for Enriques surfaces over any characteristic~\cite{liedtke2015}.
    To give some intuition for the complex result, let us prove that over the complex numbers these surfaces give at most a
    $10$-dimensional space of Enriques surfaces.

    Notice that each $Q_i$ is chosen from the same $12$-dimensional affine
    space and $S_{\bQ}$ depends only on their span, which is an element of
    $\Gr\left(3, \mathbb{C}^{12}\right)$. This is a $27$-dimensional variety.
    However, since we have fixed $\sigma$, the quadrics $Q_i$ will give an isomorphic K3 surface
    (with an isomorphic involution) if we act on $\PP^5$ by an automorphism
    that commutes with $\sigma$.
    Such automorphisms are given by block matrices in $PGL(6)$ of the form
    \begin{equation}
        C=\begin{pmatrix}A&0\\0&B \end{pmatrix}\quad\mbox{or}\quad C=\begin{pmatrix}0&A\\B&0 \end{pmatrix}
    \end{equation}
    where $A$ and $B$ are matrices in $GL(3)$, up to scaling. Thus, the space of automorphisms preserving the
    $\sigma$-invariant quadrics has dimension $2\cdot9-1=17$. Modulo these automorphisms,
    we now have a $10$-dimensional projective space of K3 surfaces with an involution.
    Note that the condition that $\bQ$ be enriquogeneous is an open condition. 

    \medskip

    We now aim at making the Enriques surfaces obtained above as $S_{\bQ}/\sigma$
    explicit. In other words we want to present them as embedded into a
    projective space.

    The first step is to identify the quotient of $\mathbb{P}^5$ by the
    involution $\sigma$. Let $S = \CC[x_0, x_1, x_2, y_0, y_1, y_2]$ be the
    homogeneous coordinate ring. Then the quotient is $\Proj
    \left(S^{\sigma}\right) = \Proj\left( \CC[x_i, y_iy_j] \right)$. The
    Enriques $S_{\bQ}$ is cut out of $\Proj\left( \CC[x_i, y_iy_j] \right)$ by
    the quadrics $\bQ$, so that
    \begin{equation}
        S_{\bQ} = \Proj\left( \CC[x_i, y_iy_j]/\bQ \right).
    \end{equation}
    This does not give us an embedding into $\mathbb{P}^8$, since the
    variables $x_i$ and $y_iy_j$ have different degrees. Rather we obtain an
    embedding into a weighted projective space $\mathbb{P}(1^3, 2^6)$.
    Therefore we replace $\CC[x_i, y_iy_j]$ by the Veronese subalgebra
    \begin{equation}
        S_{\bQ} \simeq \Proj\left( \CC[x_ix_j, y_iy_j]/\bQ \right).
    \end{equation}
    This algebra is generated by $12$ elements $x_ix_j$, $y_iy_j$ for
    $i,j=0,1,2$, so that $S_{\bQ}$ is embedded into a $\mathbb{P}^{11}$. The
    relations $\bQ$ are \emph{linear} in the variables $x_ix_j$ and $y_iy_j$, so
    that $S_{\bQ}$ is embedded into a $\mathbb{P}^8$.

    Let us rephrase this geometrically.
    Consider the second Veronese re-embedding
    ${v:\mathbb{P}^5 \to \mathbb{P}^{20}}$. The coordinates of
    $\mathbb{P}^{20}$ are forms of degree two in $x_i$ and $y_i$. The
    involution $\sigma$ extends to an involution on $\mathbb{P}^{20}$ and this
    time the invariant coordinate ring is generated by the \emph{linear forms}
    corresponding to products $x_ix_j$ and $y_iy_j$.
    Therefore the quotient is embedded into $\mathbb{P}^{11}$, which has
    coordinate ring corresponding to those $12$ forms.

    \begin{equation}
        \begin{tikzpicture}
            \matrix (m) [matrix of math nodes, row sep=2em, column sep=2.5em]
            {
                \mathbb{P}^5  & \mathbb{P}^{20} \\
                & \mathbb{P}^{11} \\
            };
            \path[->,font=\scriptsize]
            (m-1-1) edge node[auto] {$v$} (m-1-2)
                    edge node[auto,swap] {$\pi$} (m-2-2);
            \path[->,dashed,font=\scriptsize]
            (m-1-2) edge node[auto] {} (m-2-2);
        \end{tikzpicture}
    \end{equation}
    where $\pi$ denotes the quotient by the involution $\sigma$. Then
    the image $\pi(\mathbb{P}^5)$ is cut out by $12$ binomial quadrics, which
    are the $6$ usual equations between $x_ix_j$ and the $6$ corresponding equations
    for $y_iy_j$.
    It is the join of two Veronese surfaces which constitute its singular locus.
    Quadrics in $\CC[x_i, y_i]$ which have the form $F_i + G_i$ for $F_i\in
    \CC[x_i]$ and $G_i\in \CC[y_i]$ correspond bijectively to linear forms on
    the above $\mathbb{P}^{11}$.
    A choice of enriquogeneous quadrics $\bQ$ corresponds to a general choice of
    three linear forms on $\mathbb{P}^{11}$. We obtain the corresponding Enriques
    surface $S_{\bQ}$ as a linear section of $\pi(\mathbb{P}^5)$. Summing up,
    we have the following chain of inclusions
    \begin{equation}
        V \cap \pi(\PP^5) \subset \pi(\PP^5) \subset \PP(1^3, 2^6) \subset
        \PP^{11}
    \end{equation}
    where $V$ is a codimension three linear section. Note that although $V\cap
    \pi(\PP^5)$ is a complete intersection in $\pi(\PP^5)$, this is not
    contradictory to (a natural generalisation of)
    Proposition~\ref{ref:nosimpleequations:prop}, because $\pi(\PP^5)$ is
    singular. Note also that sufficiently ample embeddings of varieties are
    always cut out by quadrics, see~\cite{MumfordQuadrics, SmithSidman}, so
    this suggests that our embedding is sufficiently good.

\begin{proof}[Theorem~\ref{ref:mainthm1}]
    The surfaces obtained from enriquogeneous quadrics are arithmetically
    Cohen-Macaulay of degree $16$ as they are linear sections of $\pi(\PP^5)$
    possessing those properties. Every Enriques surface can be obtained by
    this procedure if one allows $\bQ$ not satisfying the
    smoothness condition by~\cite{C83}.\qed
\end{proof}

    Below we provide (very basic) \emph{Macaulay2} \cite{M2} code for obtaining the equations of
    $S_{\bQ}$ explicitly, using the above method.
    We could take any field as $\mathtt{kk}$; we use a finite field to take random elements.

\begin{verbatim}
kk = ZZ/1009;
P5 = kk[x0,x1,x2,y0,y1,y2];
P11 = kk[z0,z1,z2,z3,z4,z5,z6,z7,z8,z9,z10,z11];
pii = map(P5, P11, {x0^2, x0*x1, x0*x2, x1^2, 
                    x1*x2, x2^2, y0^2, y0*y1, 
                    y0*y2, y1^2, y1*y2, y2^2});
\end{verbatim}
We can verify that the kernel of \texttt{pii} is generated by 12 binomial
quadrics and has degree $16$.
\begin{verbatim}
assert(kernel pii ==
ideal(z10^2-z9*z11, z8*z10-z7*z11, z8*z9-z7*z10,
      z8^2-z6*z11,  z7*z8-z6*z10,  z7^2-z6*z9,
      z4^2-z3*z5,   z2*z4-z1*z5,   z2*z3-z1*z4,
      z2^2-z0*z5,   z1*z2-z0*z4,   z1^2-z0*z3))
assert(degree kernel pii == 16)
\end{verbatim}
Now we generate an Enriques from a random set of linear forms named \texttt{linForms}.
To see the quadrics in $\PP^5$ take \texttt{pii(linForms)}.
\begin{verbatim}
linForms = random(P11^3, P11^{-1})

randomEnriques = (kernel pii) + ideal linForms
\end{verbatim}
We now check whether it is in fact an Enriques. Computationally it is much
easier to check this for the associated K3 surface, 
since we need only check that \texttt{K3} is a smooth surface (first two assertions below) and that the involution is fixed-point-free on \texttt{K3} (last two assertions).
\begin{verbatim}
K3 = ideal pii(linForms)
assert (dim K3 == 3)
assert (dim saturate ideal singularLocus K3 == -1)
assert (dim saturate (K3 + ideal(y0,y1,y2)) == -1)
assert (dim saturate (K3 + ideal(x0,x1,x2)) == -1)
\end{verbatim}
If the \texttt{K3} passes all the assertions, then \texttt{randomEnriques} is
an Enriques surface. Its ideal is given by 12 binomial quadrics listed above
and three linear forms in \texttt{P11}.

\goodbreak
\begin{example}\label{ex:explicitinP5}
    Over $\kk = \mathbb{F}_{1009}$ the choice of
\begin{verbatim}
linForms = matrix{{2*z2+z6+5*z7+8*z11,
                   2*z0+8*z4+z9,
                   5*z1+4*z3+4*z5+6*z8}}
\end{verbatim}
in the above algorithm gives an Enriques surface.
\end{example}
Finally, we check that $\pi(\PP^5)$ is arithmetically Cohen-Macaulay. Using
\texttt{betti res kernel pii} we obtain its Betti table.
\begin{equation}
    \begin{bmatrix}
        1&\text{.}&\text{.}&\text{.}&\text{.}&\text{.}&\text{.}\\
        \text{.}&12&16&6&\text{.}&\text{.}&\text{.}\\
        \text{.}&\text{.}&36&96&100&48&9\\
    \end{bmatrix}
\end{equation}
The projective dimension of $\pi(\PP^5)$ (the number of
columns) is equal to the codimension, thus $\pi(\PP^5) \subset \PP^{11}$ is
arithmetically Cohen-Macaulay, see~\cite[Section~10.2]{Schenck_CAG}. Therefore
all its linear sections are also arithmetically Cohen-Macaulay.

\section{Analytified and tropical Enriques surfaces}\label{sec:tropicalization}

    The aim of this section is to discuss the basics of tropical and analytic
    geometry and to construct a K3 surface, whose tropicalization is nice
    enough for computations of tropical homology.
    This is done in Example~\ref{ref:kristinsK3:exam}; we obtain
    a K3 surface with an involution, which on the tropical side is the
    antipodal map.

    As an excellent reference for tropical varieties we
    recommend~\cite{Maclagan_Sturmfels_tropibook}, especially Section~6.2.
    For analytic spaces in the sense of Berkovich we recommend~\cite{Berkovich,
    Gubler_Rabinoff_Werner}.
\begin{figure}
  \centering
    \includegraphics[width=0.5\textwidth]{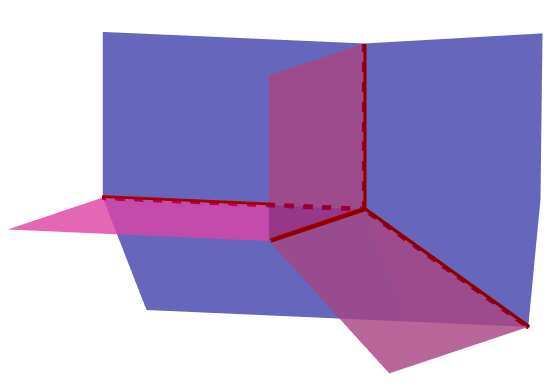}
      \caption{A tropicalization of $\PP^2$. The red segments mark the intersection lines between planes}
      \label{tropplane}
\end{figure}

    Let $\CC \subset \kk$ be a field extension and suppose that $\kk$ has a
    non-trivial valuation $\nu:\kk^* \to \mathbb{R}$ such that $\nu(\CC^*) = \{0\}$.
    Suppose further that $\kk$ is algebraically closed, so that $\nu(\kk^*)$ is dense in $\mathbb{R}$.
    Without much loss of generality one can take $\kk = \CC\{\{z\}\} =
    \bigcup_{n\in \mathbb{N}} \CC((z^{1/n}))$, the field of Puiseux series,
    with valuation yielding the lowest exponent of $z$ appearing in the series.

    For every point $p = (p_1, \ldots , p_n)\in (\kk^*)^n$ its valuation is $\nu(p) =
    (\nu(p_1), \ldots , \nu(p_n))$.

    \begin{definition}\label{ref:tropicalviavals:def}
        Let $X$ be a toric variety with torus $(\kk^*)^n$ and $Y \subset X$ be
        a closed subvariety.  The tropical variety of $Y$ is the closure of
        the set
        \begin{equation}
            \left\{ \nu(p) \ |\ p\in (\kk^*)^n \cap Y \right\} \subset
            \mathbb{R}^n,
        \end{equation}
        we denote it by $\trop(Y \subset X)$ or briefly $\trop(Y)$.
    \end{definition}
    We remark that $\trop(Y \subset X)$ is a polyhedral complex of dimension
    $\dim Y$ and has rich combinatorial structure, see~\cite[Chapter
    3]{Maclagan_Sturmfels_tropibook} and references therein.

    Now we discuss tropicalized maps.
    A morphism of tori $\varphi:(\kk^*)^n \to (\kk^*)^m$ is given by
    $\varphi = (\varphi_1,  \ldots , \varphi_m)$ where
    $\varphi_i(t) = b_i\cdot t^{a_i}$ for $i=1, \ldots ,m$.
    For each such  $\sigma$ there is a \emph{tropicalized map} $\trop(\varphi):\RR^n \to \RR^m$
    given by
    \begin{equation}\label{eq:tropmaps}
        \trop(\varphi)_i(v) = \nu(b_i) + (a_i\cdot v) \qquad i=1, \ldots ,m.
    \end{equation}
    One can check that the following diagram commutes:
    \begin{equation}
        \begin{tikzpicture}
            \matrix (m) [matrix of math nodes, row sep=2em, column sep=2.5em]
            {
                {(\kk^*)^n} & (\kk^*)^m\\
                \RR^n & \RR^m\\
            };
            \path[->,font=\scriptsize]
            (m-1-1) edge node[auto] {$\varphi$} (m-1-2)
                    edge node[auto] {$\nu$} (m-2-1)
            (m-1-2) edge node[auto] {$\nu$} (m-2-2)
            (m-2-1) edge node[auto] {$\trop(\varphi)$} (m-2-2);
        \end{tikzpicture}
    \end{equation}
    The reason for the existence of this map is that one can compute
    the valuation $\nu$ of $\varphi(t)$ by knowing only the valuation of $t$.

    A notable problem of tropical varieties is that it is known how to tropicalize a map only when it is monomial; in this sense, the naive tropicalization is not a functor.
    This problem is removed once one passes to Berkovich
    spaces.
    We will not discuss Berkovich spaces in detail: we invite the reader to see~\cite{Berkovich,
    Gubler_Rabinoff_Werner} or ~\cite{Payne_limit} for a slightly more elementary
    introduction.

    For every finite-type scheme $X$ over a valued field
    $\kk$, its Berkovich analytification $X^{an}$ is the
    analytic space (see~\cite[Chapter~3]{Berkovich}) which best approximates $X$.
    The space $X^{an}$ is locally ringed (in the usual sense, see~\cite[4.3.6]{vakil}) and there is a morphism
    $\pi:X^{an} \to X$ such that every other map
    from an analytic space  factors through $\pi$.
    If $X = \Spec A$ is affine, then the points of $X^{an}$ are in bijection with the
    multiplicative semi-norms on $A$ which
    extend the norm on $\kk$.
    Most importantly the analytification is \emph{functorial}: for every map $f:X\to Y$ we get
    an induced map
    \begin{equation}\label{eq:mapsanalitify}
        f^{an}:X^{an} \to Y^{an}.
    \end{equation}
    If $X = \Spec A$ and $Y = \Spec B$ are affine, then $f$ induces
    $f^{\#}: B\to A$ and the map $f^{an}$ takes
    a seminorm $|\cdot|$ on $A$ to the seminorm $b\to |f^{\#}(b)|$ on $B$.

    The analytification of an affine variety $X$ is
    the limit of its tropicalizations by~\cite{Payne_limit}.
    Namely let $X$ be an affine variety and consider its embeddings
    $i:X\to\mathbb{A}^n$ into affine spaces.
    For any two embeddings $i:X\to \mathbb{A}^n$ and $j:X\to \mathbb{A}^m$ and
    a toric morphism $\varphi:\mathbb{A}^n \to \mathbb{A}^m$ satisfying $j =
    \varphi\circ i$ we get by~\eqref{eq:tropmaps} a tropicalized map
    $\trop(X \subset \mathbb{A}^n) \to \trop(X \subset \mathbb{A}^m)$.
    For every embedding $X \subset \mathbb{A}^n$ there is an associated map
    \begin{equation}\label{eq:fromanal}
        X^{an} \to \trop(X \subset \mathbb{A}^n),
    \end{equation}
    which maps a multiplicative seminorm $|\cdot|$ to the associated valuation
    $-\log|\cdot |$, see~\cite[pg. 544]{Payne_limit}.
    The main result of~\cite{Payne_limit} is that the inverse limit is
    homeomorphic to the Berkovich analytification via the limit of maps
    defined in~\eqref{eq:fromanal} above. Hence one has:
    \begin{equation}
    X^{an} = \varprojlim \trop({X \subset \mathbb{A}^n}).
    \end{equation}

    We now return to the case of Enriques surfaces.
    We are interested in finding an Enriques surface $S/\sigma$ with a K3 cover
    $S$ suitable for tropicalization. Specifically we would like $\sigma$ to
    be an involution acting without fixed points on the tropical side.
    In this sense the examples obtained as in Sect. \ref{sec:constructions} are not suitable.
    \begin{example}
        Let us consider the K3 surface $S_{\bQ}$ defined using enriquogeneous
        quadrics in Section~\ref{sec:constructions} with
            $\sigma(x_0, x_1, x_2, y_0, y_1, y_2) = (x_0, x_1, x_2, -y_0,
            -y_1, -y_2)$.
            Since $\nu(-1) = 0$, the tropicalized involution
            $\trop(\sigma)$, defined by Equation~\eqref{eq:tropmaps}, is the identity map on $\RR^6$.
    \end{example}

    To obtain a K3 surface with an involution $\sigma$ tropicalizing to a
    fixed-point free involution, we consider embeddings into products of $\PP^1$.
    Consider the involution $\tau:\mathbb{P}^1 \to \mathbb{P}^1$ given by
    $\tau([x:y]) = [y:x]$ and the involution $\sigma:(\mathbb{P}^1)^3\to
    (\mathbb{P}^1)^3$ given by applying $\tau$ to every coordinate.
    The map $\tau$ restricts to the torus $\CC^*$ and is given by $\CC^* \ni t\to
    t^{-1}\in \CC^*$. Therefore $\trop(\tau)(v) = -v$ by Equation~\eqref{eq:tropmaps}.
    Consequently the tropicalization $\trop(\sigma):\mathbb{R}^3\to
    \mathbb{R}^3$ is given by
    \begin{equation}
        \trop(\sigma)(v) = -v.
    \end{equation}
    This map is non-trivial and has only one fixed point.

    \begin{example}[An example of a K3 surface with a fixed-point-free
        involution]\label{ref:kristinsK3:exam}
        Let $\kk$ be an algebraically closed field with a nontrivial valuation
        $\nu:\kk^*\to \mathbb{R}$, for example $\kk = \CC\{\{z\}\}$.

        Let $S \subset \mathbb{P}^1_{\kk} \times \mathbb{P}^1_{\kk} \times
        \mathbb{P}^1_{\kk}$ be a smooth surface given by a section of the
        anticanonical divisor of $(\mathbb{P}^1_{\kk})^3$, 
        i.e., a triquadratic polynomial.
        We remark that the Newton polytope of $S$ is the $3$-dimensional cube
        $[0, 2]^3$.
        We introduce the following assumptions on $S$:
        \begin{enumerate}
            \item $S$ is smooth;
            \item $S$ is invariant under the involution $\sigma$;
            \item the subdivision induced by $S$ on its Newton polytope $[0,
                2]^3$ is a unimodular triangulation, that is, the polytopes in
                the triangulation are tetrahedra of volume equal to
                $1/6$, see~\cite[pg. 13]{Maclagan_Sturmfels_tropibook} for
                details.
        \end{enumerate}
        Each such $S$ is a K3 surface.
        Under our assumptions the point $(0, 0, 0)$ is not in the tropical
        variety of $S$. Indeed, if it was in $\trop(S)$, that variety would not be
        locally linear at $(0, 0, 0)$. But $\trop(S)$ is coming from a
        unimodular triangulation, so it is locally linear everywhere. Hence
        $(0, 0, 0)$ is outside and so $\trop(\sigma)$ is a fixed-point-free
        involution on $\trop(S)$.
        
\begin{figure}
      \vspace{-1em}
  \centering
    \includegraphics[width=0.45\textwidth]{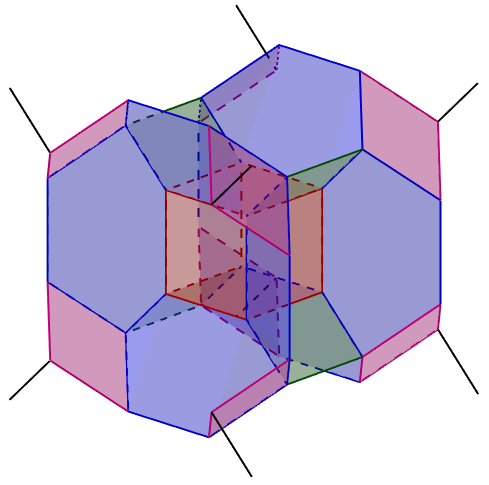}
    \caption{The bounded part of a tropical K3 in $\PP ^1\times \PP^1 \times \PP^1$}
      \label{K3}
      \vspace{-1.5em}
\end{figure}

        By Equation~\eqref{eq:mapsanalitify} map $\sigma:S\to S$
        induces also an involution $\sigma^{an}:S^{an}\to S^{an}$ which is compatible
        with $\trop(\sigma)$ under the projection $\pi$ defined in
        Equation~\eqref{eq:fromanal}; the following diagram commutes.
        \begin{equation}%
            \begin{tikzpicture}%
                \matrix (m) [matrix of math nodes, row sep=2em, column sep=2.5em]
                {
                    S^{an} &  S^{an}\\
                    \trop(S) & \trop(S)\\
                };
                \path[->,font=\scriptsize]
                (m-1-1) edge node[auto] {$\sigma^{an}$} (m-1-2)
                        edge node[auto] {$\pi$} (m-2-1)
                (m-1-2) edge node[auto] {$\pi$} (m-2-2)
                (m-2-1) edge node[auto] {$\trop(\sigma)$} (m-2-2);
            \end{tikzpicture}%
        \end{equation}
    \end{example}

\section{The tropical homology}\label{Homology}

The plan of this section is an explicit calculation of the tropical homology
of a tropical K3 surface and a tropical Enriques surface. We intend to use the
construction in Example~\ref{ref:kristinsK3:exam} in order to obtain
tropicalizations which are locally linear (locally look like
tropicalizations of linear spaces), and then compute their tropical cohomology
groups. In accordance to the results in \cite{IKMZ}, the dimensions of such
homology groups should coincide with the Hodge numbers of the surfaces
themselves. We carry out the calculation by hand for some curves, a
tropical K3 and also for an object, which we believe to be the associated tropical
Enriques. See~\cite{Kristin} for computation of tropical homology using
\emph{Polymake}.


\begin{theorem}[{\cite[Special case of
    Theorem~2]{IKMZ}}]\label{ref:homologyagrees}
    Let $X \subset \PP^N$. Suppose
    that $\trop(X) \subset \trop(\PP^N)$ has multiplicities all equal to $1$ and that it is locally
    linear. Then the tropical Hodge numbers agree with Hodge numbers of $X$:
    \begin{equation}
        \dim H_{p, q}(\trop(X)) = \dim H^{p, q}(X, \RR).
    \end{equation}
\end{theorem}
\goodbreak
For the definition of multiplicities we refer
to~\cite[Chapter~3]{Maclagan_Sturmfels_tropibook}. A tropical variety is
\emph{locally linear} (see e.g. \cite{vig}) if a Euclidean neighborhood of each point
is isomorphic to a Euclidean open subset of the tropicalization of a linear subspace $\PP^n
\subset \PP^{m}$. For example, a hypersurface in $\PP^N$ is locally linear if and only if
the subdivision of its Newton polygon is a triangulation. It has
multiplicities one if and only if this triangulation is unimodular.

Note that $X$ is not assumed to intersect the torus of $\PP^N$. Therefore this
theorem applies for example to $X \subset (\PP^1)^3 \subset \PP^7$, see
Section~\ref{sec:delPezzo}, or more generally to $X$ in any projective toric variety with
fixed embedding.

One could wonder whether Theorem~\ref{ref:homologyagrees} enables one to
identify not only dimensions but homology classes. This is possible provided
that a certain spectral sequence degenerates at the $E_2$ page.
This $E_2$ page is equal to
$H^q(X, \mathcal{F}^p)$, where $\mathcal{F}^p = \Hom(\mathcal{F}_p, \RR)$. See the discussion after
Corollary~2\ in~\cite{IKMZ} or~\cite{Clemens}.

For a more detailed introduction to tropical homology, see
e.g.~\cite{Brief_intro_to_trops, IKMZ}. We will now give generalities about
tropical homology and compute some examples of interest. In particular, we will compute the dimensions of the tropical homology groups and show how Theorem \ref{ref:homologyagrees} holds. The last part of the paper is dedicated to showing a particular instance of this theorem for a special tropical K3 with involution and for its quotient.

\medskip

Recall that $\trop(\PP^n) = \mathbb{TP}^n$ is homeomorphic to an $n$-simplex,
see~\cite[Chapter~6.2]{Maclagan_Sturmfels_tropibook} and that it is covered by
$n+1$ copies of
\begin{equation}
    \mathbb{T}^n = \trop(\A^n) = (\{-\infty\} \cup \mathbb{R})^n,
\end{equation}
which are complements of torus invariant divisors.
Let $X$ be a tropical subvariety of $\mathbb{TP}^n$. The definitions of
sheaves $\mathcal{F}_p$
and groups $C_{p, q}$ computing the homology are all local, so we assume that
$X \subset \trop(\A^n)$ is contained in one of the distinguished open subsets.
We denote by
\begin{equation}
\mathbb{T}^J= \{ x\in \mathbb{T}^n \ | \ x_i=-\infty \text{ for all } i\not \in J\}
\end{equation}
for $J\subset \{1,...,n\}$ the tropicalization of smaller torus orbits.
Let now $X\in \mathbb{T}^n$ be a polyhedral complex. The {\it sedentarity}
$I(x)$ of a point $x\in X$ is the set of coordinates of $x$ which are equal to
$-\infty$, and we set $J(x):= \{1,...,n\} \setminus I(x)$. We denote by
\begin{equation}
\RR
^{J(x)} = \RR^n/\RR^{I(x)}
\end{equation}
the interior of $\mathbb{T}^{J(x)}$. For a face $E \subset X\cap \RR ^{J(x)}$ adjacent to $x$, we let $T_x(E)\subset T_x(\RR ^{J(x)})$ be the cone spanned by the tangent vectors to $E$ starting at $x$ and directed towards $E$. Set the following:
\begin{enumerate}
\item The {\bf tropical tangent space} $\mathcal{F}_1(x) \subset T_x(\RR ^{J(x)})$ is the vector space generated by all $T_x(E)$ for all $E$ adjacent faces to $x$;
\item The {\bf tropical multitangent space} $\mathcal{F}_p(x) \subset \bigwedge ^p T_x(\RR ^{J(x)}) $ is the vector space generated by all vectors of the form $v_1\wedge ... \wedge v_p$ for vectors $v_1,...,v_p\in T_x(E)$ for all $E$ adjacent faces to $x$ (this implies $\mathcal{F}_0(x)\cong \RR $)
\end{enumerate}

One can show that the multitangent vector space $\mathcal{F}_p(x)$ for $x\in X$ only depends on the minimal face $\Delta \subset X$ containing $x$. Hence we can write $\mathcal{F}_p(\Delta) := \mathcal{F}_p(x)$ for each $x\in \Delta$. We have the following group of $(p,q)$-chains
\begin{equation}
C_{p,q}(X) := \bigoplus _{\Delta \ q-\text{dim face of }X} \mathcal{F}_p(\Delta)
\end{equation}
giving rise to the chain complex
\begin{equation}
C_{p,\bullet } = \{ \cdots \longrightarrow  C_{p,q+1}(X) \overset{\partial}{\longrightarrow} C_{p,q}(X) \overset{\partial}{\longrightarrow} C_{p,q-1}(X) \longrightarrow \cdots \}
\end{equation}
where the differential $\partial$ is
the usual simplicial differential (we choose orientation for each face) composed with inclusion maps given by $\iota
: \mathcal{F}_p(\Delta) \to \mathcal{F}_p(\Delta')$ for $\Delta'
\prec \Delta$, see examples below.

Note that even when $\Delta'$ and $\Delta$ have different
sedentarities, we have $I(\Delta') \supset I(\Delta)$ so we get a natural
map $\RR^{J(x)}  = \RR^n/\RR^{I(x)} \twoheadrightarrow \RR^n/\RR^{I(x')} =
\RR^{J(x')}$ inducing the map
$\iota: \mathcal{F}_p(\Delta) \to \mathcal{F}_p(\Delta')$.

\begin{definition}
The $(p,q)${\it -th tropical homology group} $H_{p,q}(X)$ of $X$ is the $q$-th homology group of the complex $C_{p,\bullet }$.
\end{definition}

In the light of Theorem~\ref{ref:homologyagrees}, if $X = \trop(X')$ is a
tropicalization of suitable variety $X'$, then $\dim H_{p,q}(X)$ are the Hodge
numbers of $X'$.
For all $X$ the tropical Poincar\'e duality holds: $\dim H_{d-p, d-q}(X) =
\dim H_{p, q}(X)$, see \cite{ShawPoincare}.

\begin{example}[Line]
Let us compute the tropical homology of a tropical line $L$, as in
Figure~\ref{tropline}.
\begin{description}
\item[$\underline{\mathbf{p=0}}$:] From the discussion above, one
    immediately
    sees that $C_{0,0}(L)=\RR^4$ and $C_{0, 1} = \RR^3$ injects into
    $C_{0, 0}$, thus $\dim H_{0, 0}(X) = 1$ and $H_{0, 1}(X) = 0$.
\item[$\underline{\mathbf{p=1}}$:] The chain complex is $0\to C_{1,1}(X) \to C_{1,0}(X)\to 0$. Now, the same consideration as in the previous item yields $C_{1,0}(X)= \mathcal{F}_1(v_1)=\RR \langle e_1, e_2 \rangle$, where $e_1=(-1,0)$ and $e_2=(0,-1)$ are the standard basis vectors of $\RR ^2$ up to a sign. Moreover, one has that 
    \begin{equation}
C_{1,1}(X)= \mathcal{F}_1(p) \oplus \mathcal {F}_1(q) \oplus \mathcal{F}_1(r)
= \RR \langle e_1\rangle \oplus  \RR \langle e_2\rangle \oplus  \RR \langle
- e_1 - e_2 \rangle .
\end{equation}
The differential
\begin{equation}
\RR \langle e_1\rangle \oplus  \RR \langle e_2\rangle \oplus  \RR \langle
- e_1 - e_2 \rangle \overset{\partial}{\longrightarrow}\RR \langle e_1, e_2 \rangle
\end{equation}
is given by the natural inclusion $e_1 \mapsto e_1$, $e_2 \mapsto e_2$, and $-e_1-e_2 \mapsto -e_1-e_2$. Hence the kernel of the differential above is one dimensional, generated by
the sum $\langle e_1\rangle + \langle e_2 \rangle + \langle -e_1-e_2\rangle$.
So one has $\dim H_{1,0}(X)=0 $ and $\dim H_{1,1}(X) = 1$.
\end{description}
\end{example}

\begin{figure}\label{tropline}
  \centering
    \includegraphics[width=0.35\textwidth]{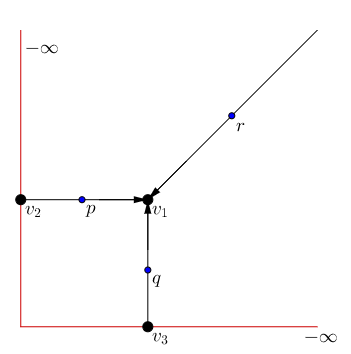}
      \caption{A tropical line}
\end{figure}

\begin{remark}\label{ref:tropicalzerosingular}
    By definition $\mathcal{F}_{0}(x) = \mathbb{R}$. Thus the complex
    $C_{0, \bullet}$ is in fact the singular homology complex for the
    subdivision of $X$ by polyhedra. Therefore the tropical homology group
    $H_{0, q}(X)$ is canonically identified with the singular homology group
    $H_{q}(X, \mathbb{R})$.
\end{remark}

\begin{example}[Elliptic curve]\label{sec:elliptic}
The next example example of tropical homology we compute is that of an elliptic curve in $\PP^1 \times \PP^1$. Its tropicalization is shown in Figure~\ref{tropell}.
\begin{figure}
  \centering
    \includegraphics[width=0.3\textwidth]{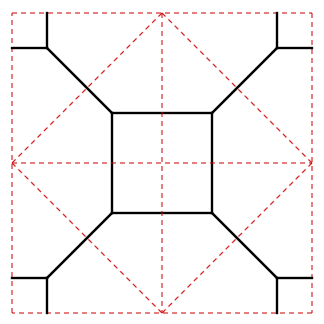}
      \caption{A tropical elliptic curve in $\PP ^1\times \PP^1$}
      \label{tropell}
\end{figure}
By the isomorphism $H_{0,q}(X)\cong H_q(X,\RR)$, it immediately follows that
$H_{0,0}(X)\cong \RR$ and $H_{0,1}(X)\cong \RR$.
We can compute $H_{1,1}(X)$ directly from the complex
\begin{equation}
C_{1,1}(X)\to C_{1,0}(X) .
\end{equation}
We get that $C_{1,1}(X)\cong \RR^E$ and $C_{1,0}(X)\cong \RR^{2V}$,
where $E = 16$ (respectively, $V = 8$) denotes the number of edges (respectively, of
interior vertices). The kernel of the map $C_{1,1}(X)\to C_{1,0}(X)$ is
$H_{1,1}(X)\cong \RR$ generated by the boundary of the square, hence $H_{1,1}(X)\cong \RR$.
\end{example}

\subsection{Del Pezzo in $(\PP^1)^3$}\label{sec:delPezzo}


Consider a surface $S$ in $\PP^1 \times \PP^1 \times \PP^1$ which is a section of
$\cO(1, 1, 1) := \cO(1) \boxtimes \cO(1) \boxtimes \cO(1)$; this is a del
Pezzo surface, its anticanonical divisor is by adjunction the restriction of $\cO(1, 1, 1)$ and
so the anticanonical degree is $6$. The equation
$F$ of $S$ can be written as
\begin{equation}\label{eq:delPezzoequation}
    F = \sum_{0\leq i,j, k\leq 1} a_{ijk} x^iy^jz^k,
\end{equation} where $x$, $y$, $z$ are
local coordinates on respective projective lines.
Suppose that we are over a valued field. Suppose further that
    $a_{ijk} = a_{1-i,1-j,1-k}$
for all indices and that $a_{1, 0, 0} > \max(a_{0, 1, 0},\, a_{0, 0, 1})$. Then the induced
subdivision of a cube is regular, as seen in Figure~\ref{tropicaldelPezzo}.

\begin{figure}
  \centering
  \begin{minipage}[b]{0.5\textwidth}
    \includegraphics[width=\textwidth]{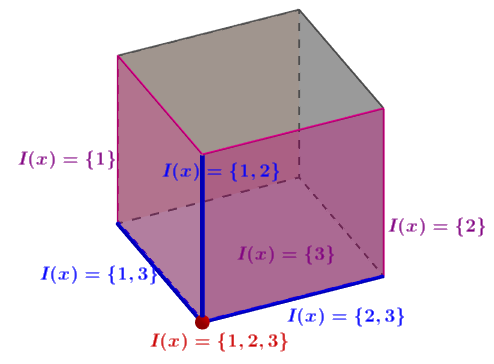}
      \caption{A tropicalization of $\PP ^1\times \PP^1\times
          \PP^1$ with the sedentarities of the faces at infinity}
      \label{tropcube}
  \end{minipage}
  \hfill
  \begin{minipage}[b]{0.4\textwidth}
    \includegraphics[width=0.9\textwidth]{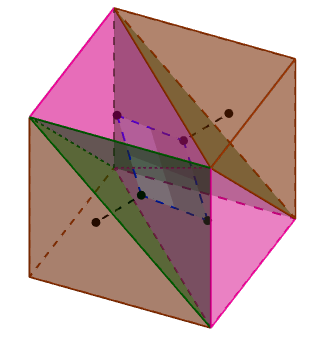}
    \caption{Regular subdivision of the cube and tropical del Pezzo}
    \label{tropicaldelPezzo}
  \end{minipage}
\end{figure}

Directly from the picture we see that there are $6$ points, $18$ edges and
$19$ faces in the non-sedentary part of $\trop(S)$. Consider now the sedentary
part. First recall that $\trop((\PP^1)^3)  \simeq (\mathbb{R}\cup \left\{
\pm\infty \right\})^3$ is homeomorphic to the cube, see Figure~\ref{tropcube}.
Its faces correspond to torus-invariant divisors in $(\PP^1)^3$.
The boundary $\trop(S) \setminus \mathbb{R}^3$ decomposes into $6$ components,
the intersections of $\trop(S)$ with those faces. We now make use of the following
\begin{theorem}[{\cite[Theorem~6.2.18]{Maclagan_Sturmfels_tropibook}}]\label{ref:closure:thm}
	Let $Y\subset T$ and let $\bar{Y}$ be the closure of $Y$ in a toric variety $X$. Then $\trop(\bar{Y})$ is the closure of $\trop(Y)$ in $\trop(X)$.
\end{theorem}
Applying Theorem~\ref{ref:closure:thm} to $Y = \bar{S}$ we derive that the boundary of the tropicalization is the tropicalization of the boundary, so we have
\begin{equation}
    \trop(S) \cap \trop(D) = \trop(S\cap D)
\end{equation}
for each torus-invariant divisor. Such $D$ is one of the divisors defined by
$x^{\pm 1}$, $y^{\pm 1}$, $z^{\pm 1}$. Without loss of generality, assume $D =
(x = 0)$. By restricting the element $F$ of~\eqref{eq:delPezzoequation} to $D$
we get $\sum_{0\leq j, k\leq 1} a_{0jk} y^jz^k$, which cut out a quadric,
whose tropicalization is given in Figure~\ref{tropQuad}. In particular, it has
five edges, two
mobile points and four sedentary points.

\begin{figure}
  \centering
    \includegraphics[width=0.3\textwidth]{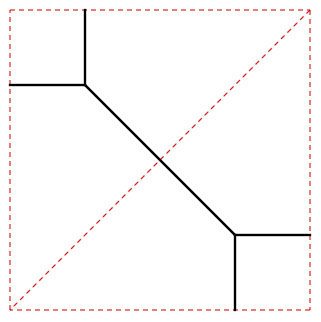}
      \caption{A tropical quadric in $\PP ^1\times \PP^1$}
      \label{tropQuad}
\end{figure}

\goodbreak
In total we get the following strata.
\begin{center}
\begin{tabular}{@{}ll ll ll ll ll@{}}
    Tropical del Pezzo&\phantom{mm}& $|sedentarity|$ &\hspace{1em}& $0$  &\hspace{1em}& $1$  &\hspace{1em}& $2$\\\midrule
    && points        && $6$ && $12$ && $12$\\
    && edges         && $18$&& $30$ && $-$\\
    && faces         && $19$ && $-$  && $-$
\end{tabular}
\end{center}
This information enables us to immediately compute the $C_{p, q}$ even without
analysing maps. This is because our del Pezzo is locally linear: near
each vertex the tropical structure looks like the tropicalization of $\mathbb{P}^2 \subset
\mathbb{P}^3$, as shown in Figure~\ref{K3}.
The complexes are
\begin{alignat}{3}
    &C_{0, 2} = \mathbb{R}^{19}          &&\to C_{0, 1} = \mathbb{R}^{18} \oplus
    \mathbb{R}^{30} &&\to C_{0, 0} = \mathbb{R}^{30}\\
    &C_{1, 2} = \mathbb{R}^{2\cdot 19}   &&\to C_{1, 1} = \mathbb{R}^{3\cdot
    18} \oplus \mathbb{R}^{30} &&\to C_{1, 0} = \mathbb{R}^{3\cdot 6} \oplus
    \mathbb{R}^{2\cdot 12}\\
    &C_{2, 2} = \mathbb{R}^{19}          &&\to C_{2, 1} =
    \mathbb{R}^{2\cdot 18} &&\to
    C_{2, 0} = \mathbb{R}^{3\cdot 6}.
    \label{eq:tropihomologydelPezzo}
\end{alignat}
By comparing $H_{0, \bullet}$ with singular homology and then using Poincar\'e
duality we get immediately that
\begin{equation}
    H_{0, 0}  \simeq  H_{2, 2}  \simeq \RR\quad H_{0, 1} = H_{0, 2} =
    H_{2, 0} = H_{2, 1} = 0.
\end{equation}
So the interesting part is the homology of $C_{1, \bullet}$.
It is not impossible to compute it by hand, however it would take a lot of
space to explain it properly, so we merely present a series of reductions, by removing strata
corresponding to higher sedentarity first. Each of these reductions corresponds
to finding an exact subcomplex $D \subset C_{1, \bullet}$ and reducing to
computing homology of $C_{1, \bullet}/D$.

Consider a sedentary point $p$ on the face of a cube. This point has two edges
$e_1$, $e_2$ going towards the boundary of this face (and a third edge, which
is irrelevant here). In $C_{1, \bullet}$ these polyhedra give a
\emph{subcomplex} $\RR[e_1]\oplus \RR[e_2]\to \RR^2[p]$, which is exact.
Thus the homology of $C_{1, \bullet}$ is the homology of
the quotient $C'$ by all these subcomplexes for $12$ choices of $p$. The
quotient is
\begin{equation}
    \RR^{2\cdot 19} \to \RR^{3\cdot 18} \oplus \RR^{6} \to \RR^{3\cdot 6}.
\end{equation}
Next, consider the picture on Figure~\ref{tropicaldelPezzo}, consider one of
the two corner vertices and all its adjacent faces ($3$ edges, $3$ faces, $1$
simplex). In the tropical variety those correspond to $1$ point $p$, $3$ edges
$e_i$ and $3$ faces $f_i$ and glue together to form on tropical
$\mathbb{A}^2$. Such an $\mathbb{A}^2$ has no higher homology and
correspondingly, the sequence
\begin{equation}
    \bigoplus \RR^2[f_i] \to \bigoplus \RR^3[e_i]\to \bigoplus \RR^3[p]
\end{equation}
is exact. Moreover it is a subcomplex of $C'$. Dividing $C'$ by two
subcomplexes given by two corner vertices, we get $C''$ equal to
\begin{equation}
    \RR^{2\cdot 13} \to \RR^{3\cdot 12} \oplus \RR^{6} \to \RR^{3\cdot 4}.
\end{equation}
In this new sequence the $\RR^{3\cdot 4}$ corresponds to $4$ multitangent
spaces at four vertices of the square in the interior, see
Figure~\ref{tropicaldelPezzo}. None of the edges adjacent to them was modified
in the process thus it is clear that the right map is
surjective. Hence $H_{1, 0} = 0$. By Poincar\'e duality we get $H_{1, 2} = 0$
and thus
\begin{equation}
\dim H_{1, 1} = 36 + 6 - 26 - 12 = 4,
\end{equation}
as expected from the Hodge diamond of a del Pezzo of anticanonical degree $6$.

\subsection{A K3 surface in $(\PP^1)^3$}\label{sec:K3}

Let $S \subset \mathbb{P}^1 \times \mathbb{P}^1 \times \mathbb{P}^1$ be a K3
surface over a valued field $\kk$ as in Example~\ref{ref:kristinsK3:exam}.
This section discusses its tropical homology and relations to its
Hodge classes; shortly speaking using tropical homology we recover the
expected Hodge numbers and an anti-symplectic involution.

Before we begin computations on the homology, let us discuss how many points,
edges and faces the polyhedral decomposition of the tropicalization has. This
polyhedral decomposition is dual to the subdivision induced on the $2\times
2\times 2$ cube by the coefficients of $S$, as explained
in~\cite[Definition~2.3.8, Figure~1.3.3]{Maclagan_Sturmfels_tropibook}.
See Figure~\ref{K3}, where the bounded part of $\trop(S)$
is presented.

Let us restrict to the torus and consider polyhedra with empty
sedentarity. First and foremost, $\trop(S)$ comes from a regular
subdivision into $48$ simplices, so it has $48$ distinguished points.
Consider now faces of the subdivision (edges in the tropicalization).
Each face may be either ``inner'', shared by two tetrahedra or ``outer''
adjacent to only one of them. There are $48$ outer faces and each tetrahedron
has four faces, thus in total there are $\frac{48\cdot 4 + 48}{2} = 120$ faces
in the subdivision. As seen in the del Pezzo case, there are $19$ edges in a
subdivision of a unit cube. In the $2\times 2\times 2$ cube we have $8\cdot 19$
of those segments; $36$ of them are adjacent to exactly two cubes, $6$ of
them are adjacent to four cubes and the others stick to one cube. Thus there
are $8\cdot 19 - 36 - 3\cdot 6 = 98$ segments.

Recall that $\trop((\PP^1)^3) = \bar{\RR}^3$, where $\bar{\RR} = \RR \cup \{\pm \infty\}$,
this is homeomorphic to a cube, see~Figure~\ref{tropcube}.
The boundary of $\trop(S)$ is the intersection of $\trop(S)$ with the
boundary of this cube. Pick a face $\mathcal{F}$ of the cube. It is the
tropicalization of one of the six toric divisors $x_i^{\pm 1}$ for $i=1, 2, 3$, say to $x_1$.
Again utilizing~Theorem~\ref{ref:closure:thm}, we have
\begin{equation}
    \trop(S) \cap \mathcal{F} = \trop(S\cap (x_1=0)).
\end{equation}
But $S\cap (x_1 = 0)$ is an elliptic curve in $\mathbb{P}^1\times
\mathbb{P}^1$ and by Section~\ref{sec:elliptic} we know that its
tropicalization has
 $16$ edges, $8$ mobile points and $8$ sedentary points, see
 Figure~\ref{tropell}.
In total we have the following strata.

\begin{center}
\begin{tabular}{@{}ll ll ll ll ll@{}}
    Tropical K3&\phantom{mm}& $|sedentarity|$ &\hspace{1em}& $0$  &\hspace{1em}& $1$  &\hspace{1em}& $2$\\\midrule
    && points        && $48$ && $48$ && $24$\\
    && edges         && $120$&& $96$ && $-$\\
    && faces         && $98$ && $-$  && $-$
\end{tabular}
\end{center}

This information enables us to immediately compute the $C_{p, q}$ even without
analysing maps. This is because $S$ is locally linear: near
each vertex the tropical structure looks like the tropicalization of $\mathbb{P}^2 \subset
\mathbb{P}^3$ see Figure~\ref{tropplane} and compare in Figure~\ref{K3}.

The complexes are
\begin{alignat}{3}
    &C_{0, 2} = \mathbb{R}^{98}          &&\to C_{0, 1} = \mathbb{R}^{120} \oplus
    \mathbb{R}^{96} &&\to C_{0, 0} = \mathbb{R}^{120}\\
    &C_{1, 2} = \mathbb{R}^{2\cdot 98}   &&\to C_{1, 1} = \mathbb{R}^{3\cdot
    120} \oplus \mathbb{R}^{96} &&\to C_{1, 0} = \mathbb{R}^{3\cdot 48} \oplus
    \mathbb{R}^{2\cdot 48}\\
    &C_{2, 2} = \mathbb{R}^{98}          &&\to C_{2, 1} =
    \mathbb{R}^{2\cdot 120} &&\to
    C_{2, 0} = \mathbb{R}^{3\cdot 48}.
    \label{eq:tropihomologyK3}
\end{alignat}
From this fact alone we see that $\chi(C_{1,\bullet}) = 2\cdot 98 - 3\cdot 120 -
  96 + 3\cdot 48 + 2\cdot = -20$ in concordance with the expected result.
  Moreover one can show that $H_{1, 0} = 0$, roughly because the
  classes of sedentary edges surject to classes of sedentary points and
  other points can be analyzed directly by Figure~\ref{K3}.
  By Poincar\'e duality, $H_{1, 2} = 0$.
  We have now
  \begin{equation}\label{eq:H11forK3}
      -20 = \chi(C_{1, \bullet}) = \dim H_{1,0} - \dim H_{1, 1} + \dim H_{1, 2} =
      -\dim H_{1, 1},
  \end{equation}
  thus $\dim H_{1, 1} = 20$.

    \medskip

   We will now consider $(0, q)$-classes.
  The homology of $C_{0, \bullet}$ is just the singular
  homology of the
  tropical variety by Remark~\ref{ref:tropicalzerosingular}. The tropical variety is
  contractible to the boundary of the cube.
  Thus $C_{0, \bullet}$ is exact in the middle and its homology groups are the
  homology groups of the sphere:
  \begin{equation}
      H_{0, 0}  \simeq \mathbb{R},\quad H_{0, 1} = 0,\quad H_{0, 2} =
      \mathbb{R}.
  \end{equation}
  
  We have just given an explicit proof of our special case of Theorem~\ref{ref:homologyagrees}.
  \begin{proposition}\label{ref:hodgeagreeK3:prop}
      The tropical Hodge numbers of $\trop(S)$ agree with the Hodge numbers of
      $S$.\qed
  \end{proposition}

  We expose an explicit generator of $H_{0, 2}$ and analyze the action of
  $\sigma$ on this space. Briefly speaking, this class is obtained as the
  boundary of the interior of the cube.

  To expand this, consider the boundary of the cube and the complex
  $C'_{2} \to C'_{1} \to C'_0$ computing its singular homology. This
  boundary can be embedded into a full cube and the complex $C'$ becomes part
  of the complex $C''$ computing the homology of the cube
  \begin{equation}
      0\to C''_3 \to C''_2 \to C''_1 \to C''_0
  \end{equation}
  Since the cube is contractible, the complex $C''$ is exact. Hence the
  unique class $\omega$ in $H^2(C')$ is the boundary of the class $\Omega$ in $C''_3$.
  Consider now the action of $\sigma$ on the $\RR^3$ containing the tropical
  variety. We have $\sigma(\mathbf{x}) = -\mathbf{x}$ in $\RR^3$, thus $\sigma$ changes
  orientation, hence $\sigma(\Omega) = -\Omega$, thus it is anti-symplectic:
  \begin{equation}\label{eq:omegadoesnotdescend}
      \sigma(\omega) = \sigma(\partial \Omega) = -\omega.
  \end{equation}
  This is expected, since otherwise $\omega$ would descend to a class in the
  tropical homology of the tropicalized Enriques surface and give a non-zero $(0, 2)$
  class on this surface.

    \medskip

  Finally we investigate $\sigma$-invariant part $C_{p, \bullet}^{\sigma}$ of
  the complexes $C_{p, \bullet}$. Since we work over characteristic different
  from two, the functor $(-)^{\sigma}$ is exact and so the homology of
  $C_{p, \bullet}^{\sigma}$ is the invariant part of the homology of
  $C_{p, \bullet}$.
  Moreover, by $\trop(S)/\sigma$ is a tropical manifold and
  $C_{\bullet, \bullet}^{\sigma}$ computes its tropical homology,
  see~\cite[Chapter~7]{Brief_intro_to_trops}.
  In particular the homology groups $H_{p, q}^{\sigma} = H^{q}(C_{p,
  \bullet}^{\sigma})$ satisfy
  $H^{\sigma}_{p, q} = H^{\sigma}_{2-p, 2-q}$.
  We believe, though
  have not proved this formally here, that the manifold $\trop(S)/\sigma$ is
  a tropicalization of the Enriques surface $S/\sigma$.
  Taking this belief for granted, the homology classes of $C_{\bullet,
  \bullet}^{\sigma}$ compute the tropical homology of Enriques surface
  $S/\sigma$.

  Let us proceed to the computation.
  First, it is straightforward to compute the dimensions of $C_{p,
  q}^{\sigma}$, since
  $\trop(S)$ does not contain the origin. Hence every face $F$ of $\trop(S)$ gets
  mapped by $\sigma$ to a unique other face $F'$ so that the action of $\sigma$
  on the space spanned by $[F]$ and $[F']$ always decomposes into invariant
  subspace $[F] + [F']$ and anti-invariant space $[F] - [F']$. Therefore $\dim
  C_{p, q}^{\sigma} = \frac{1}{2}\cdot \dim C_{p, q}$ for all $p, q$ and the
  sequences are
    \begin{alignat}{3}
        &C_{0, 2}^{\sigma} = \mathbb{R}^{49}          &&\to C_{0, 1}
        ^{\sigma}= \mathbb{R}^{60} \oplus
        \mathbb{R}^{48} &&\to C_{0, 0} ^{\sigma}= \mathbb{R}^{60}\\
        &C_{1, 2} ^{\sigma}= \mathbb{R}^{2\cdot 49}   &&\to C_{1, 1} ^{\sigma}= \mathbb{R}^{3\cdot
        60} \oplus \mathbb{R}^{48} &&\to C_{1, 0} ^{\sigma}=
        \mathbb{R}^{3\cdot 24} \oplus
        \mathbb{R}^{2\cdot 24}\\
        &C_{2, 2} ^{\sigma}= \mathbb{R}^{49}          &&\to C_{2, 1} ^{\sigma}=
        \mathbb{R}^{2\cdot 60} &&\to
        C_{2, 0} ^{\sigma}= \mathbb{R}^{3\cdot 24}.
        \label{eq:tropihomologyEnriques}
    \end{alignat}

  The key result is already computed in~\eqref{eq:omegadoesnotdescend} when considering $(0, q)$-classes: the generator
  $\omega$ of $H_{0, 2}$ does not lie in $H_{0, 2}^{\sigma}$. Thus
  $H_{0, 2}^{\sigma} = H_{0, 1}^{\sigma} = 0$ and $H_{0, 0}^{\sigma}  \simeq  \RR$.
  By symmetry $H_{2, 0}^{\sigma}  = H_{2, 1}^{\sigma} = 0$ and $H_{2,
  2}^{\sigma}
  \simeq \RR$. Finally
  \begin{equation}
      -\dim H_{1, 1} = \dim H_{1, 0} - \dim H_{1, 1} + \dim H_{1, 2} = - \chi\left(C_{1, \bullet}^{\sigma}\right) =
      -\frac{1}{2}\chi(C_{1, \bullet}) = -10,
  \end{equation}
  as in Equation~\eqref{eq:H11forK3}. We obtain the following counterpart of
  Proposition~\ref{ref:hodgeagreeK3:prop}.
  \begin{proposition}\label{ref:hodgeagreeEnriques:prop}
      The dimensions of $\sigma$-invariant parts of tropical homology groups of $S$ agree with the Hodge numbers of $S/\sigma$.\qed
  \end{proposition}

\section{Topology of analytifications of Enriques
surfaces}\label{sec:analytic}

    In this section we analyze the analytification of an Enriques surface
    which is the quotient of the K3 surface from
    Example~\ref{ref:kristinsK3:exam}.
    Fix a valued field $\kk$ and a K3 surface $S \subset (\PP^1)^3$ over $\kk$
    together with an involution $\sigma:S\to S$, as in
    Example~\ref{ref:kristinsK3:exam}.
    We first analyze the topology of $S^{an}$ itself.

    \begin{proposition}\label{ref:homotopyofSan:prop}
        The topological space $S^{an}$ has a strong deformation
        retraction onto a two-dimensional sphere $C$.
        More precisely, there exist continuous maps $s:C\to S^{an}$ and $e: S^{an}
        \to C$, so that $es = \operatorname{id}_C$ and $se$ is homotopic to
        $\operatorname{id}_{S^{an}}$. The maps $s$ and $e$ may be chosen to be
        $\sigma$-equivariant, where $\sigma$ acts on $C$ antipodally.
    \end{proposition}
    \begin{proof}
        In Section~\ref{sec:K3} we consider the tropicalization $\trop(S)
        \subset (\RR\cup \{\pm \infty\})^3$ with the antipodal involution
        $\trop(\sigma)$. We shorten $\trop(\sigma)$ to $\sigma$.
        There is a cube ${C \subset \trop(S)}$ fixed under the involution,
        see~Figure~\ref{K3}.
        This cube is a strong deformation retract of $\trop(S)$ and the
        retraction can be chosen to be $\sigma$-equivariant.
        In the following we identify $C$ with a two-dimensional sphere.

        It remains to prove that the tropical variety $\trop(S)$ is a strong deformation
        retract of $S^{an}$ under the map $\pi:S^{an}\to\trop(S)$ defined in
        Equation~\eqref{eq:fromanal}.
        We note that $\trop(S)$ is sch\"on, i.e. its intersection with every
        torus orbit is smooth (\cite[Definition~6.4.19]{Maclagan_Sturmfels_tropibook}). Moreover all
        multiplicities of top degree polyhedra are equal to one, hence the
        multiplicity at each point is equal to one by
        semicontinuity, see~\cite[Lemma~3.3.6]{Maclagan_Sturmfels_tropibook}.
        Therefore $\pi$ has a section $\trop(S)\to S^{an}$
        whose image is equal to a skeleton $S(\mathcal{S}, H)$ of a suitable
        semistable model $(\mathcal{S}, H)$ of $S$, see~\cite[Remark
        9.12]{Gubler_Rabinoff_Werner_second}.
        The skeleton $S(\mathcal{S}, H)$ is a proper strong deformation retract of
        $S^{an}$ by~\cite[\S4.9]{Gubler_Rabinoff_Werner}. The retraction map
        $S^{an} \to \trop(S)$ is equal to $\pi$, hence $\sigma$-equivariant as discussed in
        Example~\ref{ref:kristinsK3:exam}. The
        retraction $s$ in the claim of the theorem is the composition of
        retractions from $S^{an}$ to $\trop(S)$ and from $\trop(S)$ to the cube
        constructed above.\qed
    \end{proof}
    \begin{corollary}
        The analytified K3 surface $S^{an}$ is homotopy equivalent to a two-dimensional
        sphere.\qed
    \end{corollary}

    \begin{remark}\label{ref:problemremark}
        From Proposition~\ref{ref:homotopyofSan:prop} it does not follow that
        the homotopy between $se$ and $\operatorname{id}_{S^{an}}$ can be chosen
        $\sigma$-equivariantly. This is most likely true, but presently there
        seems to be no reference for this fact.
    \end{remark}

    Now we will analyze the topology of the analytification of the Enriques
    surface $S/\sigma$ using our knowledge about $S^{an}$.
    Let us briefly recall the maps which we will use.
    The quotient map $q:S\to S/\sigma$ analytifies to $q^{an}:S^{an} \to
    (S/\sigma)^{an}$.
    For any $X$ we denote $\pi:X^{an}\to X$ the natural map.
    Summarizing, we consider the following diagram.
        \begin{equation}%
            \begin{tikzpicture}%
                \matrix (m) [matrix of math nodes, row sep=2em, column sep=2.5em]
                {
                    S^{an} &  S^{an} & (S/\sigma)^{an}\\
                    S & S & S/\sigma\\
                };
                \path[->,font=\scriptsize]
                (m-1-1) edge node[auto] {$\sigma^{an}$} (m-1-2)
                        edge node[auto] {$\pi$} (m-2-1)
                (m-1-2) edge node[auto] {$\pi$} (m-2-2)
                        edge node[auto] {$q^{an}$} (m-1-3)
                (m-1-3) edge node[auto] {$\pi$} (m-2-3)
                (m-2-1) edge node[auto] {$\sigma$} (m-2-2)
                (m-2-2) edge node[auto] {$q$} (m-2-3);
            \end{tikzpicture}%
        \end{equation}
    It is crucial that $q^{an}$ is a quotient by $\sigma^{an}$, as we now prove.
    \begin{proposition}\label{ref:quotientscommutewithinvolution:prop}
        We have $(S/\sigma)^{an} = S^{an}/\sigma^{an}$ as topological spaces.
    \end{proposition}
    \begin{proof}
        First we prove the equality of \emph{sets}
        \begin{equation}\label{eq:setquotient}
            (S/\sigma)^{an} = S^{an}/\sigma^{an}.
        \end{equation}
        Consider $x\in (S/\sigma)^{an}$ and its image $\pi(x)\in S/\sigma$.
        We first describe the fiber $S^{an}_x$ of $q^{an}$ over $x$.
        Let $U = \Spec A$ be an affine neighborhood of $\pi(x)$, then the point
        $x$ corresponds to a semi-norm $|\cdot|_x$ on $A$ and $\pi(x)$ corresponds
        to the prime ideal $\mathfrak{p}_x =
        \{ f\in A\ :\ |f|_x = 0\}$, see~\cite[Remark 1.2.2]{Berkovich}.  By $\mathcal{H}(x)$ we denote the completion
        of the fraction field of $A/\mathfrak{p}_x = \kappa(\pi(x))$.
        We have the following equality~\cite[pg.~65]{Berkovich} of fibers
        \begin{equation}
            S^{an}_{x} = \left( S_x \times_{\kappa(\pi(x))} \mathcal{H}(x) \right)^{an}.
            \label{eq:fibers}
        \end{equation}
        In down-to-earth terms, the set $S^{an}_{x}$ consists of
        multiplicative seminorms on
        the $\mathcal{H}(x)$-algebra $R = H^0(S_x, \cO_{S_x})
        \otimes_{\kappa(\pi(x))} \mathcal{H}(x)$ which extend the norm
        $|\cdot|_x$ on $\mathcal{H}(x)$.
        Using \cite[Proposition~1.3.5]{Berkovich} we may assume $\mathcal{H}(x)$ is
        algebraically closed.
        Since $H^0(S_x,\cO_{S_x})^{\sigma} =
        \kappa(\pi(x))$, we have $R^{\sigma} = \mathcal{H}(x)$.
        Similarly, the ring $R$ is a rank two free $\mathcal{H}(x)$-module.
        Then $R$ is isomorphic to
        either $\mathcal{H}(x)^{\times 2}$ with $\sigma$ permuting
        the coordinates or to $\mathcal{H}(x)[\varepsilon]/\varepsilon^2$.
        Consider a multiplicative seminorm $|\cdot|_y$ on $R$. Its kernel $\mathfrak{q} =
        \{f\in R\ :\ |f|_y = 0\}$ is a prime ideal in $R$ and in both cases above
        we have $R/\mathfrak{q} =\mathcal{H}(x)$. Since $|\cdot|_y$ agrees with
        $|\cdot|_x$ on $\mathcal{H}(x)$, we see that $|\cdot|_y$ is determined
        uniquely by its kernel. The involution $\sigma$ acts transitively on
        those, hence $\sigma^{an}$ acts transitively on the set $S^{an}_x$ and
        Equality~\eqref{eq:setquotient} is proven.

        Second, we prove that $(S/\sigma)^{an} = S^{an}/\sigma^{an}$ as topological spaces;
        in other words that the topology on $(S/\sigma)^{an}$ is induced from this of
        $S^{an}$. Take an open subset $U \subset S^{an}$. We want to show that
        $q^{an}(U)$ is open. Clearly $U \cup \sigma^{an}(U) \subset S^{an}$ is open and
        a union of fibers, so its complement $Z \subset S^{an}$ is closed and a
        union of fibers. Now the map $q^{an}$ is finite~\cite[3.4.7]{Berkovich}, thus
        proper and so closed~\cite[3.3.6]{Berkovich}. In particular $q^{an}(Z)
        \subset (S/\sigma)^{an}$ is
        closed, so $\pi(U) = (S/\sigma)^{an} \setminus q^{an}(Z)$ is
        open. This proves that $(S/\sigma)^{an} = S^{an}/\sigma^{an}$ as topological
        spaces.\qed
    \end{proof}

    \begin{corollary}\label{ref:retractionofenriques:cor}
        There exists a retraction from $(S/\sigma)^{an}$ onto $\RR\PP^2$. In
        particular $(S/\sigma)^{an}$ is not contractible.
    \end{corollary}
    \begin{proof}
        The argument follows formally from
        Proposition~\ref{ref:homotopyofSan:prop} and
        Proposition~\ref{ref:quotientscommutewithinvolution:prop}.
        Recall from Proposition~\ref{ref:homotopyofSan:prop} the
        $\sigma$-invariant map
        $e:S^{an} \to C$ and its section $s:C\to S^{an}$.
        Here $C$ is a two dimensional sphere with an antipodal involution $\sigma$ and clearly $C/\sigma  \simeq \RR\PP^2$.
        We now produce equivalents of $s$ and $e$ on the level of $S^{an}/\sigma^{an}
        \simeq  (S/\sigma)^{an}$.
        \begin{equation}
            \begin{tikzpicture}%
                \matrix (m) [matrix of math nodes, row sep=2em, column sep=2.5em]
                {
                    S^{an} &  S^{an} & (S/\sigma)^{an}\\
                    C & C & C/\sigma = \RR\PP^2\\
                };
                \path[->,font=\scriptsize]
                (m-1-1) edge node[auto] {$\sigma^{an}$} (m-1-2)
                        edge node[auto] {$e$} (m-2-1)
                (m-1-2) edge node[auto] {$e$} (m-2-2)
                        edge node[auto] {$q^{an}$} (m-1-3)
                (m-1-3) edge node[auto] {$e$} (m-2-3)
                (m-2-1) edge node[auto] {$\trop(\sigma)$} (m-2-2)
                (m-2-2) edge node[auto] {$q$} (m-2-3)
                (m-2-1) edge[bend left=15] node[auto] {$s$} (m-1-1)
                (m-2-2) edge[bend left=15] node[auto] {$s$} (m-1-2)
                (m-2-3) edge[bend left=15] node[auto] {$s$} (m-1-3);
            \end{tikzpicture}%
        \end{equation}

        The map $qe:S^{an} \to s(C)/\sigma = \RR\PP^2$
        is such that $qe\circ \sigma^{an} = qe$, thus
        by definition of quotient and by
        Proposition~\ref{ref:quotientscommutewithinvolution:prop} it induces a unique map
        $e:S^{an}/\sigma^{an} = (S/\sigma)^{an}\to
        \RR\PP^2$. Similarly $q^{an} \circ s$ satisfies $q^{an}\circ s \circ \trop(\sigma) =
        q^{an}\circ s$,
        hence induces a unique map $s:\RR\PP^2 \to (S/\sigma)^{an}$.
        Then $e\circ s:\RR\PP^2\to \RR\PP^2$ is the
        unique map induced $\sigma$-invariant map $qes = q$. Therefore $e\circ s =
        \operatorname{id}_{\RR\PP^2}$ and so $se$ is a retraction of
        $(S/\sigma)^{an}$ onto $s(\RR\PP^2)  \simeq \RR\PP^2$.\qed
    \end{proof}
    \begin{remark}
        If the difficulty presented in Remark~\ref{ref:problemremark} was
        removed, a similar argument would show that $(S/\sigma)^{an}$ strongly
        deformation retracts onto $\RR\PP^2$.
    \end{remark}

    \begin{proof}[of
        Theorem~\ref{ref:topologyOfAnalytifications:theorem}]
        Follows from Proposition~\ref{ref:homotopyofSan:prop} and
        Corollary~\ref{ref:retractionofenriques:cor}.\qed
    \end{proof}

\section*{Conclusion}
We constructed an explicit Enriques surface as the quotient $S/\sigma$ by an involution
$\sigma$ on a K3 surface $S$.  We then tropicalized $S$ and considered the quotient
$\trop(S)/\sigma$ under the tropicalized involution.
We computed in Proposition~\ref{ref:hodgeagreeEnriques:prop} the tropical
homology of this quotient and obtained expected results. We have \emph{not} proved
that $\trop(S)/\sigma = \trop(S/\sigma)$. However, we obtained this equality on the level of analytic spaces:
in Proposition~\ref{ref:quotientscommutewithinvolution:prop} we proved that
$S^{an}/\sigma^{an} = (S/\sigma)^{an}$ and concluded that $(S/\sigma)^{an}$ retracts onto an $\RR\PP^2
\subset (S/\sigma)^{an}$.

\begin{acknowledgement}
    This article was initiated during the Apprenticeship
      Weeks (22 August-2 September 2016), led by Bernd
      Sturmfels, as part of the Combinatorial Algebraic
      Geometry Semester at the Fields Institute for Research in Mathematical Sciences.
		We thank Kristin Shaw for many helpful conversations and for suggesting
	Example~\ref{ref:kristinsK3:exam}. We thank Christian Liedtke for many useful
	remarks and suggesting Proposition~\ref{ref:nosimpleequations:prop}.  We thank
	Julie Rana for discussions and providing the sources for the introduction.
    We thank Walter Gubler, Joseph Rabinoff and Annette Werner for sharing
    their insights.
	We also thank Bernd Sturmfels and the anonymous referees for providing many interesting suggestions
	and giving deep feedback. BB was supported by the Fields Institute for Research in Mathematical Sciences.
    CH was supported by the Fields Institute for Research in Mathematical Sciences and by the Clay Mathematics Institute and by NSA award H98230-16-1-0016.
    JJ was supported by the Polish National Science Center, project
	2014/13/N/ST1/02640.
\end{acknowledgement}

\end{document}